\documentclass{aptpub}
\authornames{Piet Groeneboom} 
\shorttitle{Convex hulls} 
\usepackage{amsmath,amsfonts,here,latexsym,amssymb,amscd,natbib,graphicx,pictexwd,pstricks,pst-plot}
\def\bi{\mbox{Bi}}
\def\a{\alpha}

\def\R{\mathbb R}

\def\P{{\mathbb P}}

\def\labda1{\lambda_1}
\def\labda2{\lambda_2}

\def\e{\varepsilon}

\def\t{\tau}

\def\s{\sigma}

\def\comment#1{\relax}

\def\=in{\mathop{\rm =}}

\renewcommand{\P}{{\mathbb P}}

\renewcommand{\a}{\alpha}

\numberwithin{equation}{section}

\begin{document}
\title{Convex hulls of uniform samples from a convex polygon}

\authorone[Delft University of Technology]{Piet Groeneboom}

\addressone{Delft University, DIAM, Mekelweg 4, 2628CD Delft, The Netherlands}
\footnotetext[1]{Dedicated to the memory of Alexander Nagaev.}

\begin{abstract}
In \citet{gr:88} a central limit theorem for the number of vertices $N_n$ of the convex hull of a uniform sample from
the interior of convex polygon was derived. To be more precise, it was shown that $\{N_n-\tfrac23r\log
n\}/\{\tfrac{10}{27}r\log n\}^{1/2}$ converges in law to a standard normal distribution, if $r$ is the number of
vertices of the convex polygon from which the sample is taken. 

In the unpublished preprint \citet{nagaev_kham:91} a central limit result for the joint distribution of $N_n$ and $A_n$ is given,
where $A_n$ is the area of the convex hull, using a coupling of the sample process near the border of the polygon with a Poisson point
process as in \citet{gr:88}, and representing the remaining area in the Poisson approximation as a union of a doubly infinite sequence of
independent standard exponential random variables.

We derive this representation from the representation in \citet{gr:88} and also prove
the central limit result of \citet{nagaev_kham:91}, using this representation. The relation between the variances of the asymptotic
normal distributions of number of vertices and the area, established in \citet{nagaev_kham:91}, corresponds to a relation between the actual
sample variances of $N_n$ and $A_n$ in \citet{buchta:05}. We show how these asymptotic results all follow from one simple
guiding principle. This corrects at the same time the scaling constants in \citet{cabogr:94} and \citet{nagaev:95}.
\end{abstract}


\section{Introduction}
\label{sec:intro}
\setcounter{equation}{0}
Let $N_n$ be the number of vertices of the convex hull of a sample of size $n$, drawn uniformly from the interior of a convex polygon with
$r$ vertices. It was shown in \citet{gr:88} that
$$
\{N_n-\tfrac23r\log n\}/\{\tfrac{10}{27}r\log n\}^{1/2}\stackrel{\cal D}\longrightarrow {\cal N}(0,1),
$$
where ${\cal N}(0,1)$ denotes the standard normal distribution. This was proved by coupling the sample point process near the boundary of the
convex polygon with a Poisson point process, and showing that the relevant part of the sample process could be approximated sufficiently
closely by the coupled Poisson point process. The central limit result for $N_n$ was subsequently derived from a corresponding result for
the boundary of the convex hull of the approximating Poisson point process.
These methods were also applied to the area $A_n$ of the convex hull in \citet{cabogr:94}, but unfortunately the central
limit result $A_n$ contained a scaling error (see Remark \ref{remark:nagaev}).

\citet{nagaev_kham:91}, using the coupling of (part of the) sample point process with a Poisson process introduced in \citet{gr:88},
derived the following interesting central limit theorem for the joint distribution of the number of vertices and the area of the convex hull of a
uniform sample of
$n$ points on the interior of a convex polygon.

\begin{thm}
\label{theorem:N_n_A_n}
\mbox{\rm (Theorem 1 of \citet{nagaev_kham:91})}
Let $N_n$ denote the number of vertices of the convex hull of a uniform sample of size $n$ from the interior of a
convex polygon $C$  with $r\ge3$ vertices and area $A(C)$.
Moreover, let $A_n$ denote the area of the convex hull of the sample, and let the scaled ``remaining area" $\bar{A}_n$ be defined by
$$
\bar{A}_n=n\left\{A(C)-A_n\right\}/A(C)
$$
Then
\begin{equation}
\left(\tfrac{10}{27}r\log n\right)^{-1/2}
\left(N_n-\tfrac23r\log n,\bar{A}_n-\tfrac23r\log
n\right)\stackrel{{\cal D}}\longrightarrow {\cal N}(0,\Sigma),
\end{equation}
where ${\cal N}(0,\Sigma)$ denotes the normal distribution with expectation the zero vector and covariance matrix $\Sigma$ given by
$$
\Sigma=\left(\begin{array}{ll}
1 &1\\
1 &\displaystyle{\tfrac{14}{5}}
\end{array}
\right)
$$
\end{thm}

This is an extension of the central limit theorem for the number of vertices $N_n$ in \citet{gr:88}, and one
indeed recovers the central limit theorem given there by specializing the above result to the first coordinate. Unfortunately, the preprint
\citet{nagaev_kham:91}, containing this result, was never published. Moreover, it is written in Russian and its length is 50 pages, which
might also not have helped its spread in the scientific world.

In a private correspondence Christian Buchta revealed to me that the constant for the central limit
theorem for the second component (the remaining area) in \citet{nagaev_kham:91} was consistent with a relation he had derived himself
between the finite sample variances of $N_n$ and $\bar{A}_n$.

It is the purpose of the present note to give a simple proof of Theorem \ref{theorem:N_n_A_n}, deriving the result from the central limit
theorem for $N_n$ in \citet{gr:88}.
We  think that using the central limit theorem of \citet{gr:88} considerably simplifies the proof of Theorem \ref{theorem:N_n_A_n}
in \citet{nagaev_kham:91} and perhaps more clearly reveals the beauty of their idea. The relation between the variances in Theorem
\ref{theorem:N_n_A_n} can be considered to be a precursor (in an asymptotic sense) of the relation found between the finite sample
variances in \citet{buchta:05}.

For recent work on central limit theorems for random polytopes, see, e.g., \citet{bar1:10} and \citet{bar2:10}, where also references to earlier work in this area can be found.

\section{Representation of the remaining area by i.i.d.\ exponentials}
\label{sec:exponentials}
\setcounter{equation}{0}
We consider the Poisson point process ${\cal P}$ of intensity $1$ in $\R_+^2$, and its left-lower convex hull, as in \citet{gr:88}. To make the connection with
\citet{gr:88}, we first restate the definition of the process of vertices $\{W(a):a\in\R_+\}$ consisting of the vertices of the (left-lower) convex hull of a Poisson
process ${\cal P}$ with intensity 1 in $\R_+^2$.

\begin{figure}[!ht]
\begin{center}
\strut\input{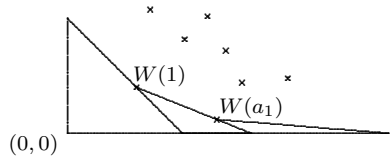}
\caption{$W(a)$-process}
\label{fig:start}
\end{center}
\end{figure}

\begin{definition}
\label{def:W(A)}
{\rm
For each $a>0$, $W(a)=(U(a),V(a))$ is the point of the realization of the Poisson process ${\cal P}$ on $\R_+^2$ such that all points of the
realization of ${\cal P}$ lie to the right of the line of the line $x+ay=c$ which passes through $W(a)$. If there are several of such points (which
happens with probability zero for fixed $a$), we define $U(a)$ ($V(a)$) as the supremum (infimum) of $x$-coordinates ($y$-coordinates) of points of this
type.
}
\end{definition}

We now have the following result (see also Theorem 2.1 of \citet{nagaev:95} for a result of this type).

\begin{thm}
\label{theorem:decomposition}
Let $a_0=1$,  let $a_1,a_2,\dots$ be the jump times of the process $\{W(a):a\ge1\}$, and let $D_0$ be the area of the
isosceles triangle $T_0$ with a basis, running through $W(1)$, and two equal sides along the $x$- and $y$-axis, meeting at the top at
the origin. Moreover, let
$D_i,\,i\ge1$, be the area of the triangle $T_i$, with top at $W(a_{i-1})$, basis along the $x$-axis, and sides along the lines
$x+a_{i-1}y=U(a_i)+a_iV(a_i)$ and
$x+a_iy=U(a_i)+a_iV(a_i)$, where $W(a_i)$, $U(a_i)$ and $V(a_i)$ are defined as in Definition \ref{def:W(A)}. Then
\begin{enumerate}
\item[(i)] The areas $D_0,D_1,\dots$ form an i.i.d.\ sequence of standard exponential random variables.
\item[(ii)] Let $S_i$ be the length of the line segment, connecting $W(a_{i-1})$ and $W(a_i)$, and let $L_i$ be the length of the
segment, obtained by extending the line segment from $W(a_{i-1})$ to $W(a_i)$ until it crosses the $x$-axis. Then the random variables
$S_i^2/L_i^2$, $i=1,2,\dots$ form an i.i.d.\ sequence of Uniform$(0,1)$ random variables, independent of $W(1)$. Moreover,
the $S_i^2/L_i^2$ are independent of the sequence $D_0,D_1,\dots$
\end{enumerate}
\end{thm}

\begin{proof}
(i). By Part (i) of Lemma 2.4 of
\citet{gr:88} we have, for $z\ge0$,
$$
P\left\{D_0>z\right\}=P\left\{\tfrac12\{U(1)+V(1)\}^2>z\right\}=\int_{\left\{(x,y):\tfrac12(x+y)^2>z\right\}}e^{-\tfrac12(x+y)^2}\,dx\,dy=e^{-z},
$$
showing that $D_0$ has a standard exponential distribution. Let ${\cal F}_a$ denote the $\s$-algebra, generated by the points $\{W(b),\,1\le b\le a\}$.
Then, as shown in \citet{gr:88}, the process of points $\{W(a),\,a\ge1\}$ is a Markov process w.r.t.\ the filtration $\{{\cal F}_a,\,a\ge1\}$. Now note
that, if $i\ge1$, $D_i>z$ exactly when there are no points in the triangle of area $z$, with top at $W(a_i)$, basis along the $x$-axis, and sides along
the lines
$x+a_{i-1}y=U(a_i)+a_iV(a_i)$ and $x+a_iy=U(a_i)+a_iV(a_i)$. Since this event is independent of the location of the points $W(a_0),\dots, W(a_{i-1})$, by
the Poisson property of the point process in $\R_+^2$, we get:
$$
P\left\{D_i>z\right\}=e^{-z},\,z\ge0,
$$
where the event $D_i>z$ is independent of $D_0,\dots,D_{i-1}$ (note that we can use the strong Markov property here).\\
(ii). The jump measure $M(a,w;\cdot)$ of the process $\{W(a):a>0\}$ is given by
\begin{equation}
\label{jump_measure}
M(a,w;B)=\int_0^yu1_B(au,-u)\,du,
\end{equation}
see (2.22) of \citet{gr:88}. Hence, conditioning on $W(a)=W(a_{i-1})=(x,y)$ and the event that there is a jump at time $a$, the location of the next
vertex has a density proportional to $u$ (representing the distance of $W(a)$ to the next vertex). So we get, for $z\in(0,1)$,
\begin{align*}
&P\left\{S_i^2/L_i^2<z\Bigm|W(a)>W(a-)=(x,y)\right\}\\
&=P\left\{S_i<L_i\sqrt{z}\Bigm|W(a)>W(a-)=(x,y)\right\}\\
&=P\left\{S_i<y\sqrt{z\left(1+a^2\right)}\Bigm|W(a)>W(a-)=(x,y)\right\}\\
&=\frac2{y^2\{1+a^2\}}\int_0^{y\sqrt{z(1+a^2)}}u\,du=z,
\end{align*}
where we use that $\tfrac12y^2\{1+a^2\}$ is the total measure of the jump measure on the line segment of length $y\sqrt{1+a^2}$, connecting $(x,y)$ and $(x+ay,0)$.
This implies that $S_i^2/L_i^2$ has a uniform distribution, in accordance with Theorem 2.1 of \citet{nagaev:95}. Moreover, since the distribution neither
involves the value of $a=a_i$ nor that of  $W(a_{i-1})$, the sequence of variables $S_i^2/L_i^2$ is i.i.d. For the same reason the variable $S_i^2/L_i^2$ is independent of $D_j$, $j\le i$. It is also seen that $S_i^2/L_i^2$ is independent of $D_j$, $j>i$, since the conditional distribution of $D_{i+1}$, given $W(a_i)$, is standard exponential, independently of the value of $W(a_i)$.
\end{proof}

\begin{corollary}
\label{cor:V(a)}
Let the sequences $a_0,a_1,\dots$ and $V(a_0),V(a_1),\dots$ be defined as in Theorem \ref{theorem:decomposition}, and let $\t_i=V(a_i)/V(a_{i-1})$,
$i=1,2,\dots$. Then the sequence of random variables  $\t_1,\t_2,\dots$ is i.i.d.\ and
$$
(1-\t_i)^2\sim \mbox{\rm Uniform}(0,1).
$$
Moreover, the random variables $\t_i$ are independent of $V(a_0)=V(1)$ and the areas $D_i$, where $D_i$ is defined as in Theorem \ref{theorem:decomposition}.
\end{corollary}

\begin{proof}
This follows from part (ii) of Theorem \ref{theorem:decomposition} since
$$
1-\t_i=1-\frac{V(a_i)}{V(a_{i-1})}=\frac{V(a_{i-1})-V(a_i)}{V(a_{i-1})}=\frac{S_i}{L_i},\,i=1,\dots,
$$
where the last equality is the proportionality relation, well-known from elementary geometry.
\end{proof}

\vspace{0.3cm}
The following result is the key to Theorem \ref{theorem:N_n_A_n}.

\begin{corollary}
\label{cor:asymp_independence}
Let, for $m=2,3\dots$, $N(1,m)$ be the number of jumps of the process $\{W(a):a\in[1,m]\}$ and let $[EN(1,m)]$ be the largest integer smaller than or equal to $EN(1,m)$. Then:
\begin{enumerate}
\item[(i)]
$$
EN(1,m)=\tfrac13\log m,
$$
\item[(ii)] As $m\to\infty$ the bivariate random variable
$$
\left(\{N(1,m)-EN(1,m)\}/\sqrt{\tfrac5{27}\log m},\sum_{i=1}^{[EN(1,m)]}(D_i-1)/\sqrt{EN(1,m)}\right)
$$
converges in distribution to a bivariate normal distribution with expectation zero and covariance matrix equal to the identity matrix $I$.
\end{enumerate}
\end{corollary}

\begin{proof}
(i). This is part (i) of Theorem 2.4 of \citet{gr:88}, which is a simple consequence of the fact that the expected jump rate of the process $\{W(a):a\ge1\}$ is given by $1/(3a)$.\\
(ii). The area $D_i$ of the triangle $T_i$, as defined in Theorem \ref{theorem:decomposition}, is given by:
\begin{equation}
\label{recursion}
D_i=\tfrac12V(a_{i-1})(V(a_{i-1})+a_iV(a_{i-1})-V(a_{i-1})-a_{i-1}V(a_{i-1}))=\tfrac12V(a_{i-1})^2(a_i-a_{i-1}).
\end{equation}
Define
$$
U_i=U(a_i),\,V_i=V(a_i),\mbox{ and }W_i=\left(U_i,V_i\right),\,i=0,1,\dots
$$
It is clear that (\ref{recursion}) gives a tridiagonal system for solving $a_i$ in terms of the $D_i$ and $V_i$. We get, using
$a_0=1$,
\begin{equation}
\label{a_n}
a_n=1+2\sum_{i=1}^n\frac{D_i}{V_{i-1}^2}\,,\,n\ge1.
\end{equation}
We now define, for $n\ge1$,
$$
Y_n=V_{n-1}^2\left\{1+2\sum_{i=1}^n\frac{D_i}{V_{i-1}^2}\right\}=V_{n-1}^2a_n.
$$
Thus,
\begin{equation}
\label{a_n-representation}
\log a_n=-2\log V_{n-1}+\log Y_n,
\end{equation}
and hence we get the ``switching relation":
\begin{equation}
\label{switch}
N(1,m)\ge n \Longleftrightarrow a_n\le m \Longleftrightarrow -2\log V_{n-1} +\log Y_n\le \log m.
\end{equation}
By Corollary \ref{cor:V(a)}:
\begin{align}
\label{expectation_V_m}
E V_n^2=EV_0^2\prod_{i=1}^n\t_i^2=6^{-n}EV_0^2,\qquad E\left(\frac{V_n^2}{V_k^2}\right)=\prod_{i=k+1}^n E\t_i^2=
6^{-(n-k)},\,n>k\ge0.
\end{align}
Since, by Theorem \ref{theorem:decomposition}, the $\t_i$ are also independent of the $D_i$, we obtain, for all $k\ge1$,
\begin{align*}
EY_n&=6^{-(n-1)}EV_0^2+2\sum_{j=1}^n E\left(\frac{V_{n-1}^2}{V_{j-1}^2}\right)=
6^{-(n-1)}EV_0^2+2\sum_{j=1}^{n-1} 6^{-j}\\
&\le 6^{-(n-1)}EV_0^2+2\sum_{j=1}^{\infty}6^{-j}.
\end{align*}
This implies, by Markov's inequality,
$$
Y_n=O_p(1),\,n\to\infty.
$$
Since we also have $Y_n\ge 2D_n$, for all $n\ge1$, where $D_n$ has a standard exponential distribution, we obtain from
this:
\begin{equation}
\label{Y_m_bounded}
\left|\log Y_n\right|=O_p(1),\,n\to\infty.
\end{equation}
We now get from (\ref{a_n-representation}):
$$
\frac{\log a_n-3n}{\sqrt{5n}}=\frac{-2\log V_{n-1} +\log Y_n-3n}{\sqrt{5n}}=\frac{-2\log V_{n-1}-3n}{\sqrt{5n}}+O_p\left(n^{-1/2}\right),
$$
as $n\to\infty$. Moreover, since
\begin{equation}
\label{repr_V_k}
-2\log V_{n-1}=-2\sum_{i=1}^{n-1}\log\left(\frac{V_i}{V_{i-1}}\right) -2\log V_0=-2\sum_{i=1}^{n-1}\log\t_i -2\log V_0,
\end{equation}
we get by the central limit theorem:
\begin{equation}
\label{asymp_norm_a_m}
\frac{\log a_n-3n}{\sqrt{5n}}=\frac{-2\sum_{i=1}^{n-1}\log\t_i-3n}{\sqrt{5n}}+o_p(1)\stackrel{{\cal D}}\longrightarrow {\cal N}(0,1),\,n\to\infty,
\end{equation}
where ${\cal N}(0,1)$ denotes the standard normal distribution.

Let
$$
B_1(m)=\sum_{i=1}^{[EN(1,m)]}(D_i-1)/\sqrt{EN(1,m)},
$$
and 
$$
B_2(m)=\{N(1,m)-EN(1,m)\}/\sqrt{\tfrac5{27}\log m},
$$
and let, for fixed $y\in\R$, $n=n_{m,y}\in\mathbb N$ be defined by:
\begin{equation}
\label{def_n_m}
n=\left[EN(1,m)+y\sqrt{\tfrac5{27}\log m}\right],\,m\to\infty.
\end{equation}
Then we find, using (\ref{switch}) and (\ref{asymp_norm_a_m}), as $m\to\infty$,
\begin{align*}
\label{N(1,a)_asymptotics}
&\P\left\{B_1(m)\ge x,\,B_2(m)\ge y\right\}
=\P\left\{B_1(m)\ge x,\,N(1,m)\ge EN(1,m)+y\sqrt{\tfrac5{27}\log m}\right\}\nonumber\\
&\sim\P\left\{B_1(m)\ge x,\,N(1,m)\ge n\right\}
=\P\left\{B_1(m)\ge x,\,\log a_n\le \log m\right\}\nonumber\\
&=\P\left\{B_1(m)\ge x,\,\frac{\log a_n-3n}{\sqrt{5n}}\le \frac{\log m-3n}{\sqrt{5n}}\right\}\nonumber\\
&\sim\P\left\{B_1(m)\ge x,\,\frac{-2\sum_{i=1}^{n-1}\log\t_i-3n}{\sqrt{5n}}\le \frac{\log m-3EN(1,m)-y\sqrt{\tfrac5{3}\log m}}{\sqrt{5n}}\right\}\nonumber\\
&\sim \P\left\{B_1(m)\ge x\right\}\P\left\{\frac{-2\sum_{i=1}^{n-1}\log\t_i-3n}{\sqrt{5n}}\le -\frac{y\sqrt{\tfrac5{3}\log m}}{\sqrt{\tfrac5{3}\log m}}\right\}\nonumber\\
&=\P\left\{B_1(m)\ge x\right\}\P\left\{\frac{-2\sum_{i=1}^{n-1}\log\t_i-3n}{\sqrt{5n}}\le -y\right\}.
\end{align*}
where we use part (i), (\ref{def_n_m}) and and Corollary \ref{cor:V(a)} (independence of the $\t_i$ and the $D_i$) in the next to last line. Since, by (\ref{asymp_norm_a_m}),
\begin{align*}
\P\left\{\frac{-2\sum_{i=1}^{n-1}\log\t_i-3n}{\sqrt{5n}}\le -y\right\}\to \Phi(-y)=1-\Phi(y),
\end{align*}
where $\Phi$ is the standard normal distributon function, the result now follows.
\end{proof}

\section{The central limit theorem}
\label{sec:Nagaev}
\setcounter{equation}{0}
In this section we prove a 2-dimensional central limit theorem, by combining the results of the preceding section with the results in \citet{gr:88}.

\begin{thm}
\label{th:CL_for_N_D}
Let $N(a,b)$ be the number of jumps in the interval $[a,b]$ of the process $W$, as defined in Definition
\ref{def:W(A)}, and let $D(a,b)$ be the area of the union of the triangles $T_i$, corresponding to points of jump $a_i\in[a,b]$, as defined in Theorem
\ref{theorem:decomposition}. Then:
$$
\left(\tfrac5{27}\log(b/a)\right)^{-1/2}\left(N(a,b)-\tfrac13\log(b/a),D(a,b)-\tfrac13\log(b/a)\right)\stackrel{{\cal D}}\longrightarrow
N(0,\Sigma),\,b/a\to\infty,
$$
where $N(0,\Sigma)$ is a bivariate normal distribution with expectation $0$ and covariance matrix defined by
\begin{equation}
\label{def_sigma}
\Sigma=\left(\begin{array}{ll}
1 &1\\
1 &\displaystyle{\tfrac{14}{5}}
\end{array}
\right)
\end{equation}
\end{thm}

\begin{proof}
As shown by the transformation to a stationary process (2.27) in \citet{gr:88}, the distribution of $N(a,b)$ only depends on the ratio
$b/a$. The same construction shows that the distribution of $D(a,b)$ only depends on the ratio $b/a$.
So we only have to prove the result for $a=1$ and $b>1$.

We know from Theorem 2.4 in \citet{gr:88} that $E N(1,a)$ $=\tfrac13\log a$ and $\mbox{var}(N(1,a))$ $\sim (5/27)\log a$, as $a\to\infty$. Moreover,
$$
D(1,a)=\sum_{a_i\in[1,a]}D_i=\sum_{a_i\in[1,a]} \mbox{area}(T_i),
$$
where the $T_i$ are the triangles of Theorem \ref{theorem:decomposition}. So we can consider $D(1,a)$ as a random sum of standard exponential random
variables, where the number of terms in the sum is equal to the random variable $N(a,b)$. Reasoning heuristically, as in the case of a compound Poisson distribution, we
would get
$$
E(D(1,a))=E N(1,a)=\tfrac13\log a,
$$
and
$$
\mbox{var}(D(1,a))=E N(1,a)+\mbox{var}(N(1,a))\sim\tfrac13\log a+\tfrac5{27}\log a=\tfrac{14}{27}\log a.
$$
We now show that we can prove the result by using this heuristic idea.

We write $D(1,a)-\tfrac13\log a$ as the sum of the terms $A_1(a)$ and $A_2(a)$, where
$$
A_1(a)=\sum_{i=1}^{[EN(1,a)]} D_i-\tfrac13\log a,
$$
defining $[EN(1,a)]$ as the largest integer not exceeding $EN(1,a)=\tfrac13\log a$, and
$$
A_2(a)=\left\{\begin{array}{ll}
\sum_{i=[EN(1,a)]+1}^{N(1,a)}D_i,\,&\mbox{ if }N(1,a)>[EN(1,a)]\\
&\\
-\sum_{i=N(1,a)+1}^{[EN(1,a)]}D_i,\,&\mbox{ if }N(1,a)\le [EN(1,a)].
\end{array}
\right.
$$
We now have, if $N(1,a)>[EN(1,a)]$
\begin{align*}
\sum_{i=[EN(1,a)]+1}^{N(1,a)}D_i=\sum_{i=[EN(1,a)]+1}^{N(1,a)}(D_i-1)+N(1,a)-[EN(1,a)],
\end{align*}
and similarly, if $N(1,a)\le[EN(1,a)]$,
\begin{align*}
-\sum_{i=N(1,a)+1}^{[EN(1,a)]}D_i=-\sum_{i=N(1,a)+1}^{[EN(1,a)]}(D_i-1)+N(1,a)-[EN(1,a)],
\end{align*}
where both sides are zero if $N(1,a)=[EN(1,a)]$. Hence we can write:
\begin{align*}
D(1,a)-\tfrac13\log a=A_1(a)+N(1,a)-[EN(1,a)]+R(a),
\end{align*}
where
$$
R(a)=\left\{\begin{array}{ll}
\sum_{i=[EN(1,a)]+1}^{N(1,a)}(D_i-1),\,&\mbox{ if }N(1,a)>[EN(1,a)]\\
&\\
-\sum_{i=N(1,a)+1}^{[EN(1,a)]}(D_i-1),\,&\mbox{ if }N(1,a)\le [EN(1,a)].
\end{array}
\right.
$$

Fix $\e>0$. 
By Theorem 2.4 in \citet{gr:88} there exists an $M=M(\e)>0$ and an $a_0=a_0(M)$ so that
$$
\P\left\{\left|\frac{N(1,a)-[EN(1,a)]}{\sqrt{\log a}}\right|> M\right\}<\e,\, a\ge a_0.
$$
Define
$$
n_{-}(a)=[EN(1,a)]-M\sqrt{\log a},\,\qquad n_{+}(a)=[EN(1,a)]+M\sqrt{\log a}\,.
$$
Then, by Doob's inequality,
\begin{align*}
&\P\left\{\max_{m\in\left[[E N(1,a)]+1,n_+(a)\right]}\left|\sum_{i=[E N(1,a)]}^{m}(D_i-1)\right|>\e\sqrt{\log a}\right\}\\
&\qquad+\P\left\{\max_{m\in\left[n_-(a),[E N(1,a)]\right]}\left|\sum_{i=m}^{[EN(1,a)]}(D_i-1)\right|>\e\sqrt{\log a}\right\}\\
&\le \frac{n_{+}(a)-n_{-}(a)+1}{\e^2\left(\log a\right)}\sim\frac{2M}{\e^2\sqrt{\log
a}}\to0,\,a\to\infty.
\end{align*}
These relations imply: $R(a)/\sqrt{\log a}=o_p(1)$, $a\to\infty$, and hence:
\begin{align}
\label{decomposition_D}
\frac{D(1,a)-E N(1,a)}{\sqrt{\log a}}
&=\frac{\sum_{i=1}^{[E N(1,a)]}(D_i-1)}{\sqrt{\log a}}
+\frac{N(1,a)-[E N(1,a)]}{\sqrt{\log a}}+o_p(1).
\end{align}
The result now follows from Corollary \ref{cor:asymp_independence} and Theorem 2.4 in \citet{gr:88}.
\end{proof}

\vspace{0.3cm}
Using the methods from \citet{gr:88} in going from the Poisson approximation to the sample process, one can now easily deduce the
central limit result Theorem \ref{theorem:N_n_A_n} from Theorem \ref{th:CL_for_N_D}. The latter method is also used in \citet{nagaev_kham:91}.

\begin{rem}
{\rm Instead of working directly with relation (\ref{recursion}), expressing the differences between successive slopes of the convex hull in
terms of the area of the corresponding rectangle and the $y$-coordinate of vertex at the intersection of the line segments with these
slopes,
\citet{nagaev_kham:91} write this relation first in the following form:
\begin{equation}
\label{recursion2}
D_i=\tfrac12V(a_{i-1})^2\left(\frac{U(a_i)-U(a_{i-1})}{V(a_{i-1})-V(a_i)}-\frac{U(a_{i-1})-U(a_{i-2})}{V(a_{i-2})-V(a_{i-1})}\right),
\end{equation}
and then deduce a recursive relation for the $U(a_i)$ in terms of the $V(a_i)$ and $D_i$ from this. They then define the random time
$$
\theta_T=\inf\left\{i:U(a_i)\ge T\right\},
$$
and consider sums of the form $\sum_{i=1}^{\theta_T}D_i$.
This seems to lead to more complicated proofs.
}
\end{rem}

\begin{rem}
\label{remark:nagaev}
{\rm
The scaling constants for the central limit theorem for the area in \citet{cabogr:94} are not correct, although a correct application of the methods
used in that paper would lead to the central limit theorem for the area, which is part of the central limit theorem \ref{theorem:N_n_A_n} above. We here tried to present the results of the unpublished preprint \citet{nagaev_kham:91} in an easily understandable way, where the presentation is considerably simplified by the use of martingales, Doob's inequality and the results from \citet{gr:88}. In view of this simpler
approach, and also the fact that Theorem \ref{theorem:N_n_A_n} is in fact a stronger (2-dimensional) result,
this approach seems preferable to the approach in \citet{cabogr:94}. On the other hand, the computations along the lines of \citet{cabogr:94} give precise information on the first and second moments, as shown below in section \ref{sec:simulation}.

Although \citet{nagaev:95} hints at the proof of the central limit theorem \ref{theorem:N_n_A_n}, there are many important missing steps, which
have to be traced down to the unpublished preprint \citet{nagaev_kham:91}. It seems fair to say that without knowledge of this preprint,
deducing the result from \citet{nagaev:95} is pretty hard. Moreover, \citet{nagaev:95} contains in the crucial relation (3.7) an incorrect scaling
constant (the constant
$5/4$ there should be $20/27$), which further complicates the derivation of Theorem
\ref{theorem:N_n_A_n}. For this reason we gave a simplified and self-contained treatment above.
}
\end{rem}

\begin{rem}
{\rm \citet{buchta:05} gives the following relation between the sample variances of $N_n$ and $\bar{A}_n$ (using the notation of Theorem 
\ref{theorem:N_n_A_n}):
$$
\frac{(n+1)(n+2)\mbox{\rm var}(\bar{A}_n)}{n^2}=\mbox{\rm var}(N_n)+d_{n+2},
$$
where
$$
d_n=(E N_n)^2-\frac{n\left(EN_{n-1}\right)^2}{n-1}-(2n-1)EN_n +2nEN_{n-1}\sim EN_n\sim\tfrac95\mbox{\rm var}(N_n),\,n\to\infty.
$$
Hence
$$
\mbox{\rm var}(\bar{A}_n)\sim \tfrac{14}5\mbox{\rm var}(N_n),\,n\to\infty,
$$
in accordance with the covariance matrix $\Sigma$ in Theorem 1 in \citet{nagaev_kham:91} (Theorem \ref{theorem:N_n_A_n} above). Note that
the split-up of the variance of
$\bar{A}_n$ corresponds to the split-up (\ref{decomposition_D}) above, where $d_{n+2}$ corresponds to the variance of the exponentials $\xi_i$ in
(\ref{decomposition_D}) and $\mbox{\rm var}(N_n)$ corresponds to the variance of the second term on the right-hand side of 
(\ref{decomposition_D}).

Theorem 2 of \citet{buchta:03} gives for the number of vertices $N_n$ of the convex hull of the points $(0,1)$, $(1,0)$ and $P_1,\dots,P_n$, where $P_1,\dots,P_n$ is a uniform sample from the interior of the triangle with vertices $(0,0)$, $(0,1)$ and $(1,0)$:
$$
EN_n=\tfrac13\left\{2\sum_{i=1}^n\frac1{i}+1\right\},
$$
and
$$
\mbox{var}\left(N_n\right)=\tfrac1{27}\left\{10\sum_{i=1}^n\frac1{i}+12\sum_{i=1}^n\frac1{i^2}-28+\frac{12}{n+1}\right\}.
$$
This gives:
\begin{equation}
\label{Buchta_relations}
EN_n\sim\tfrac23\log n,\qquad\mbox{var}\left(N_n\right)\sim\tfrac{10}{27}\log n,\qquad n\to\infty,
\end{equation}
which corresponds to the distribution results derived in \citet{gr:88}, as is also noted in \citet{buchta:03}.

The results in \citet{gr:88} and \citet{nagaev_kham:91} only imply that one gets a normal limit distribution for the number of vertices of the convex hull of a uniform sample from the interior of a convex polygon with $r$ vertices by centering with $\tfrac23r\log n$ and dividing by $(\tfrac{10}{27}r\log n)^{1/2}$. It is not proved there that the variance of the number of vertices itself is also of order $\tfrac{10}{27}r\log n$. In principle one could have a central limit theorem where the scaling needed to get the central limit result is different from what one gets from the actual variance.

However, the only thing that still seems needed to go from (\ref{Buchta_relations}) to the result that the variance itself is also of order $\tfrac{10}{27}r\log n$ seems the appropriate use of the  independence of what happens in the corners of the polygons, so that one can conclude that the variance is the sum of the variances of the number of vertices in these corners. Moreover, one has to go from what happens in the triangle to what happens in the corners of the polygon.
This is the subject of current research by Buchta. Results for higher moments of the convex hull of a uniform sample from triangle with vertices $(0,0)$, $(0,1)$ and $(1,0)$ are given in \citet{buchta:11}.
}
\end{rem}

\section{Simulations}
\label{sec:simulation}
\setcounter{equation}{0}
Let $N(a,b)$ and $D(a,b)$ be defined as in Theorem \ref{th:CL_for_N_D}. The distribution of these random variables only depends on the ratio $b/a$ and in this section we present some simulation results for these random variables, taking $a=1$ and replacing $b$ by $a$.

The algorithm, given in section 4 of \citet{nagaev:95}, was used to simulate part of the boundary of the convex hull of a Poisson process with intensity 1 in the first quadrant. The starting triangle is bounded by the $x$-axis, $y$-axis and a
line of the form $x+y=c$, where $c>0$. Its area $D_0$ has a standard exponential distribution and the point $W(1)$ is uniformly distributed on the line segment which is the hypotenuse of this triangle. 

With the algorithm of \citet{nagaev:95} we can now generate the points $W(a)$, $a\ge1$, and simulate in this way the distribution of $N(1,a)$ and
$D(1,a)$. We start with $N(1,a)$ and recall the exact expressions for the expectation $EN(1,a)$ and $\mbox{\rm var}(N(1,a))$ from
\citet{gr:88}, Theorem 2.4:
\begin{equation}
\label{EN(1,a)}
E N(1,a)=\tfrac13\log a,
\end{equation}
and
\begin{equation}
\label{var(N(1,a))}
\mbox{\rm var}(N(1,a))=\frac5{27}\log a+\frac49\left(\tan^{-1}\left(\sqrt{a-1}\right)\right)^2
+\frac89\left\{\frac{\tan^{-1}\left(\sqrt{a-1}\right)}{\sqrt{a-1}}-1\right\}.
\end{equation}
As noted on top of page 34 in \citet{cabogr:94}, the formula for the variance of $(N(1,a))$, given in Theorem 2.1 of \citet{gr:88} contained a typo (the argument of the first $\tan^{-1}$ above was $a$ instead of $\sqrt{a-1}$), and the correct formula is in fact given on p.\ 365 of \citet{gr:88} (which we use here). Note that these are exact expressions for $E N(1,a)$ and $\mbox{var}(N(1,a))$ and not asymptotic ones.

The following table shows the means and variances for 10,000 simulations for the values $\log a=10,50$ and $100$.
The exact values are given in 4 decimals accuracy.

\bigskip
\par\noindent
{\bf Table 1.}  Comparison of $EN(1,a)$ and $\mbox{Var}(N(1,a))$ with simulated and asymptotic values.
\medskip
\begin{center} 
\begin{tabular}{|r ||c |c |c | c | c |} \hline
$\log a$        &  simulated            
   &  exact
  & simulated       &  exact & asymptotic\\
&$EN(1,a)$   &$EN(1,a)$ &$\mbox{Var}(N(1,a))$   &$\mbox{Var}(N(1,a))$& $\mbox{Var}(N(1,a))$   \\ 
 \hline
 10  &\ 3.3519 & \ 3.3333  &\  2.1193  &\ 2.0596  &\ 1.8519\\
 50  &\ 16.6668 & \ 16.6667  &\  9.5908  &\  9.4670  &\ 9.2593\\
 100 &\ 33.4259 & \ 33.3333  &\ 18.7039  &\ 18.7263  &\ 18.5185\\
\hline
\end{tabular} 
\end{center}

\vspace{0.4cm}
It is seen from Table 1 that $EN(1,a)$ and $\mbox{Var}(N(1,a))$ are quite close to the simulated values and that, not unexpectedly, for $a=10$ the exact
expression for the variance of $N(1,a)$, given by (\ref{var(N(1,a))}), is closer to the simulated value than the asymptotic value.

We similarly did 10,000 simulations for the values $\log a=10,50$ and $100$ to simulate the behavior of $D(1,a)$.
Using the (corrected) methods of computation of \citet{cabogr:94} (details are given in \citet{piet:11b}), it can be shown that
$$
ED(1,a)=\tfrac13\log a,
$$
and, defining $\a=a-1$, that:
\begin{align*}
&\mbox{\rm var}(D(1,a))\\
&=\frac{14}{27}\log a+\frac2{3\a^2}+\frac4{9\a}-\frac{44}{45}
-\frac{2\{3+\a(3-4\a)\}\tan^{-1}\left(\sqrt{\a}\right)}{9\a^{5/2}}
+\frac49\left(\tan^{-1}\left(\sqrt{\a}\right)\right)^2.
\end{align*}
These are again exact expressions for $ED(1,a)$ and $\mbox{var}(D(1,a))$ and not asymptotic ones. We get the following results.

\vspace{0.4cm}
\par\noindent
{\bf Table 2.}  Comparison of $ED(1,a)$ and $\mbox{Var}(D(1,a))$ with simulated and asymptotic values.
\medskip
\begin{center} 
\begin{tabular}{|r ||c |c |c | c | c |} \hline
$\log a$        &  simulated            
   &  exact
  & simulated       &  exact & asymptotic\\
&$ED(1,a)$   &$ED(1,a)$ &$\mbox{Var}(D(1,a))$   &$\mbox{Var}(D(1,a))$& $\mbox{Var}(D(1,a))$   \\ 
 \hline
 10   &\ 3.3664 & \ 3.3333  &\ 5.4089  &\ 5.3040  &\ 5.1852\\
 50  &\ 16.6576 & \ 16.6667 &\  26.1452  &\  26.0448  &\ 25.9259\\
 100 &\ 33.4933 & \ 33.3333  &\ 52.3304  &\ 51.9707  &\ 51.8519\\
\hline
\end{tabular} 
\end{center}

\vspace{0.4cm}
We finally turn our attention to relation (3.7) in \citet{nagaev:95}. This relation gives asymptotic expressions for the expectation and variance of the number $\nu_t$ of
vertices falling in a disk $S_t$ with radius $t$ and center $(0,0)$. On the basis of the results in \citet{gr:88}, it is to be expected that
\begin{equation}
\label{points_in_disk}
E\nu_t\sim \tfrac43\log t,\qquad \mbox{var}\left(\nu_t\right)\sim \tfrac{20}{27}\log t,\,t\to\infty,
\end{equation}
whereas relation (3.7) in Nagaev (1995) gives the above relation for $E\nu_t$, but $(5/4)\log t$ as the asymptotic expression for $\mbox{var}\left(\nu_t\right)$.
The argument for (\ref{points_in_disk}) is that, first of all, $\nu_t$ can be expected to behave asymptotically as the number of vertices with coordinates $x>y$ such that $x<t$ plus the
number of vertices with coordinates $y\ge x$ such that $y<t$, since vertices with large $x$-coordinates will with high probability be very close to the $x$-axis and 
vertices with large $y$-coordinates will with high probability be very close to the $y$-axis. Secondly, again by \citet{gr:88}, the number of vertices with
coordinates $x>y$ such that $x<t$ will behave asymptotically as $N(1,t^2)$, and similarly, the number of vertices with
coordinates $y\ge x$ such that $y<t$ will behave asymptotically as $N(1/t^2,1)$.

By the construction of the algorithm in \citet{nagaev:95}, we can simulate the number of vertices $W(a)$, $a\ge1$, satisfying
$U(a)^2+V(a)^2<t^2$, by running the algorithm till we get a vertex $W(a)$ such that
$$
U(a)^2+V(a)^2\ge t^2.
$$
The resulting asymptotic behavior of $E\nu_t$ and $\mbox{Var}(\nu_t)$ is obtained from this by multiplying the
results by the factor
$2$. The table below shows the result
for 10,000 simulations for the values $\log t=10,50$ and $100$.

\bigskip
\par\noindent
{\bf Table 3.}  Comparison of $E\nu_t$ and $\mbox{Var}(\nu_t)$ with simulated and asymptotic values.
\medskip
\begin{center} 
\begin{tabular}{|r ||c |c |c | c | c |} \hline
$\log t$        &  simulated            
   &  exact
  & simulated       &$(20/27)\log t$  &$(5/4)\log t$ \\
&$E\nu_t$   &$E\nu_t$ &$\mbox{Var}(\nu_t)$   &  &   \\ 
 \hline
 10   &\ 13.0778 & \ 13.3333  &\ 7.2630  &\ 7.40741  &\ 12.5\\
 50  &\ 66.4792 & \ 66.6667  &\  37.6192  &\  37.0370  &\ 62.5\\
 100 &\ 133.1330 & \ 133.3333  &\ 74.542  &\ 74.0741  &\ 125\\
\hline
\end{tabular} 
\end{center}

\vspace{0.3cm}
Table 3 clearly suggests that the factor $5/4$ is much too large and that the correct approximation is indeed
given by (\ref{points_in_disk}) above.

\section{Concluding remarks}
\label{sec:conclusion}
There is a remarkable analogy between the behavior of the left-lower convex hull of the Poisson point process, discussed above, and the least concave majorant of (one-sided) Brownian motion without drift, as analyzed in \citet{piet:83}. In the same way there is an analogy between the behavior of the lower convex hull of the Poisson point process inside a parabola, as analyzed in \citet{gr:88} and \citet{nagaev:95}, and the least concave majorant of Brownian motion with a parabolic drift, as studied in \citet{piet:89} and \citet{piet:11a}. Why this is the case is still somewhat of a mystery and deserves (in my view) further investigation.

\vspace{0.5cm}
\noindent
\acks I want to thank Tomasz Schreiber for sending me the unpublished preprint \citet{nagaev_kham:91} and Christian Buchta for making me aware of \citet{buchta:03} and sending me the preprint \citet{buchta:11}.

\end{document}